\documentclass[a4paper,10pt]{amsart}

\usepackage{amsmath}
\usepackage{amsthm}
\usepackage{amssymb}
\usepackage{amsfonts}
\usepackage{enumitem} 
\usepackage{mathrsfs} 
\usepackage{amscd}
\usepackage{xcolor}
\usepackage[bookmarks=true,hyperindex,pdftex,colorlinks,citecolor=blue]{hyperref}%
\usepackage[a4paper,lmargin=3cm,rmargin=3cm,tmargin=4cm,bmargin=4cm,marginparwidth=2.8cm,marginparsep=1mm]{geometry} 

\DeclareMathOperator{\dist}{dist}                           
\DeclareMathOperator{\lspan}{span}                          
\DeclareMathOperator{\conv}{conv}                           
\DeclareMathOperator{\supp}{supp}                           

\DeclareMathOperator{\ext}{ext}                             
\DeclareMathOperator{\Lip}{Lip}                             
\DeclareMathOperator{\lip}{lip}                             

\newcommand{\NN}{\mathbb{N}}                                
\newcommand{\RR}{\mathbb{R}}                                

\newcommand{\abs}[1]{\left|{#1}\right|}                     

\newcommand{\pare}[1]{\left({#1}\right)}                    
\newcommand{\set}[1]{\left\{{#1}\right\}}                   
\newcommand{\norm}[1]{\left\|{#1}\right\|}                  
\newcommand{\dual}[1]{{#1}^\ast}                            
\newcommand{\duality}[1]{\left<{#1}\right>}                 
\newcommand{\cl}[1]{\overline{#1}}                          
\newcommand{\restrict}{\mathord{\upharpoonright}}           
\newcommand{\lipfree}[1]{\mathcal{F}({#1})}                 
\newcommand{\lipnorm}[1]{\norm{#1}_L}                       
\newcommand{\meas}[1]{\mathcal{M}({#1})}                    
\newcommand{\wt}[1]{\widetilde{#1}}                         
\newcommand{\bwt}[1]{\beta\wt{#1}}                          
\newcommand{\pp}{\mathfrak{p}}                              
\newcommand{\opr}[1]{\mathcal{M}_{\mathrm{op}}(#1)}         
\newcommand{\ucomp}[1]{#1^\mathcal{U}}                    

\renewcommand{\leq}{\leqslant}
\renewcommand{\geq}{\geqslant}
\newcommand{\ep}{\varepsilon}

\DeclareMathOperator{\Mol}{Mol}                              

\makeatletter
\newcommand{\labeltext}[2]{%
  \@bsphack
  \csname phantomsection\endcsname 
  \def\@currentlabel{#1}{\label{#2}}%
  \@esphack
}
\makeatother

\allowdisplaybreaks

\theoremstyle{plain}
\newtheorem{theorem}{Theorem}[section]
\newtheorem{lemma}[theorem]{Lemma}
\newtheorem{corollary}[theorem]{Corollary}
\newtheorem{proposition}[theorem]{Proposition}

\newtheorem{question}{Question}

\theoremstyle{definition}
\newtheorem*{definition*}{Definition}
\newtheorem{definition}[theorem]{Definition}
\newtheorem{example}[theorem]{Example}

\theoremstyle{remark}
\newtheorem{remark}[theorem]{Remark}

\begin{document}

\title[Extreme points and Choquet theory of free spaces]{A solution to the extreme point problem and other applications of Choquet theory to Lipschitz-free spaces}


\author[R. J. Aliaga]{Ram\'on J. Aliaga}
\address[R. J. Aliaga]{Instituto Universitario de Matem\'atica Pura y Aplicada,
Universitat Polit\`ecnica de Val\`encia,
Camino de Vera S/N,
46022 Valencia, Spain}
\email{raalva@upv.es}

\author[E. Perneck\'a]{Eva Perneck\'a}
\address[E. Perneck\'a]{Faculty of Information Technology, Czech Technical University in Prague, Th\'akurova 9, 160 00, Prague 6, Czech Republic}
\email{perneeva@fit.cvut.cz}

\author[R. J. Smith]{Richard J. Smith}
\address[R. J. Smith]{School of Mathematics and Statistics, University College Dublin, Belfield, Dublin 4, Ireland}
\email{richard.smith@maths.ucd.ie}

\date{}


\begin{abstract}
We prove that every element of a Lipschitz-free space admits an expression as a convex series of elements with compact support. As a consequence, we conclude that all extreme points of the unit ball of Lipschitz-free spaces are elementary molecules, solving a long-standing problem. We also deduce that all elements of a Lipschitz-free space with the Radon-Nikod\'ym property can be expressed as convex integrals of molecules. Our results are based on a recent theory of integral representation for functionals on Lipschitz spaces which draws on classical Choquet theory, due to the third named author.
\end{abstract}



\maketitle

\section{Introduction and preliminaries}
\label{sec:introduction}

This paper is concerned with the isometric theory of Lipschitz-free spaces. Precise definitions will follow in Section \ref{sec:lipschitz-free}. For the time being, we consider a complete metric space $M$ and the Banach space $\Lip_0(M)$ of real-valued Lipschitz functions on $M$, endowed with the best Lipschitz constant as a norm. This space admits a predual $\lipfree{M}$, the \emph{Lipschitz-free space} over $M$, that can be understood as a canonical linearisation of the metric space $M$. Lipschitz-free spaces have a close connection to the theory of optimal transport (see e.g. \cite[Section 1.2]{APS1} and \cite{OO1}), but our focus is on Banach space theory, where their main applications are often stated to be in nonlinear Banach space geometry. Many of these applications, which are isomorphic in nature, are already to be found in the seminal papers by Godefroy and Kalton \cite{GK,Kalton04}; see \cite{Godefroy_survey} for a survey on that topic. More recently, Lipschitz-free spaces have also found use as a source of isometric examples and counterexamples in linear Banach space theory; see for instance \cite[Section 2]{AALMPPV1} or \cite{Basset}.

The space $\lipfree{M}$ and its bidual $\Lip_0(M)^*$ consist of functionals on a function space, so one could hope that these functionals can be represented in some capacity using measures. There is no known general representation theorem for $\lipfree{M}$, only in particular cases such as $M=\RR^n$ \cite{CKK2}. The best explicit representation to date, valid for any $M$, can be traced back to to K. de Leeuw \cite{deLeeuw}. It consists of expressing elements of $\dual{\Lip_0(M)}$ as measures on the compactification $\bwt{M}$, where $\wt{M}$ is the space of pairs of different points in $M$. The function to be integrated against such a measure is not the Lipschitz function $f$ itself, but rather the mapping $\Phi f:(x,y)\mapsto (f(x)-f(y))/d(x,y)$. Thus we interpret functionals on $\Lip_0(M)$ as ``integrals of incremental quotients''. Every element of $\dual{\Lip_0(M)}$ admits such representing measures, called \emph{(De Leeuw) representations}, but they are not unique. Furthermore, the relation between properties of the functional and those of its representations is not straightforward, and it is challenging to distinguish which measures correspond to functionals in $\lipfree{M}$. See Section \ref{sec:de leeuw} for basic facts on De Leeuw representations.

Because representations are not unique, it is desirable to try to find the ``best'' one(s) for a given functional, whatever that may mean. In the previous works \cite{Aliaga,APS1,APS2}, it was shown that De Leeuw representations can be chosen to satisfy some optimality properties, such as having minimal norm or minimal support (more accurately, minimal \emph{shadow} -- see Section \ref{sec:de leeuw}), and that information about the structure of $\lipfree{M}$ and $\Lip_0(M)^*$ can be gleaned from the study of such representations. The very recent paper \cite{Smith24} by the third named author takes a great leap forward by defining a quasi-order on the space of representations that is designed to compare their complexity in more subtle ways. The quasi-order is induced by a certain function cone, so the resulting theory has many similarities to classical Choquet theory (see e.g. \cite[Chapter 3]{lmns10}). A brief introduction to this Choquet-like theory is provided in Section \ref{sec:choquet}. Representations that are minimal with respect to this quasi-order retain many of the previous optimality properties while gaining additional ones, such as being concentrated on ``nice'' sets or having mutually singular marginals. For the purposes of this paper, the standout property of minimal representations of elements of $\lipfree{M}$ is that they can be described exclusively using coordinates in $M$, whereas arbitrary representations might need to use coordinates in a compactification of $M$; see Theorem \ref{th:opt_conc} for a precise statement.

In this paper, we derive some consequences of this Choquet theory for the structure of Lipschitz-free spaces. In measure spaces, it is easy to decompose an element into an $\ell_1$ sum of other elements, simply by considering restrictions. For our main result, we are able to transfer that principle to Lipschitz-free spaces and decompose every element of $\lipfree{M}$ as a convex sum of elements with compact support (Theorem \ref{th:compact_decomposition}). This makes it possible to reduce some isometric questions on Lipschitz-free spaces to the case where $M$ is compact. Perhaps the most significant one is the characterisation of the extreme points of the unit ball of $\lipfree{M}$. This question has been open since Weaver's work thirty years ago \cite{Weaver95} and has been the object of vigorous research efforts in the last decade \cite{Aliaga,AliagaGuirao,AP_rmi,APPP_2020,APS1,APP,GPPR,GPAZ,Weaver2}; see Section \ref{sec:extreme} for a more detailed account. Theorem \ref{th:extreme} provides a final answer to that problem, relying on the known compact case \cite{Aliaga}. We use this to obtain a Banach-Stone type theorem for certain Lipschitz-free spaces (Theorem \ref{th:banach-stone}). Another question concerns the elements of $\lipfree{M}$ that are \emph{convex integrals of molecules}, meaning that they have a minimal-norm representation concentrated on $\wt{M}$ \cite{APS1}. We show in Theorem \ref{th:p1u_cim} that every element of $\lipfree{M}$ has that form whenever $\lipfree{M}$ has the Radon-Nikod\'ym property, again by reducing to the known compact case \cite{APS2}. This, in turn, provides information on Lipschitz functions that attain their Lipschitz constants (Corollary \ref{cr:p1u sna}). We end the paper with another convex decomposition result: every element of $\lipfree{M}$ can be split into a convex integral of molecules and another element whose representations have the complete opposite behaviour (Theorem \ref{th:diagonal decomposition}).

\subsection{Preliminaries on Lipschitz-free spaces}
\label{sec:lipschitz-free}

In what follows, $M$ will denote a complete metric space. We will tacitly assume that an arbitrary point $0\in M$ has been selected as a base point. The Lipschitz space $\Lip_0(M)$ is then the space of all Lipschitz functions $f:M\to\RR$ such that $f(0)=0$. It becomes a Banach space when endowed with the norm given by the Lipschitz constant
$$
\lipnorm{f} = \sup\set{\frac{f(x)-f(y)}{d(x,y)} \,:\, x\neq y\in M} .
$$
For each $x\in M$, the evaluation functional $\delta(x):f\mapsto f(x)$ belongs to the dual $\Lip_0(M)^*$; in fact, $\delta:M\to\Lip_0(M)^*$ is an isometric embedding. The closed subspace $\lipfree{M} = \cl{\lspan}\,\delta(M)$ is usually called the \emph{Lipschitz-free space} over $M$. The main reference for this class of spaces is \cite{Weaver2} (where they are called \emph{Arens-Eells spaces}). The assumption that $M$ is complete incurs no loss of generality as the Lipschitz and Lipschitz-free spaces over any metric space agree with those over its completion, but the precise statements of some of the results below really depend on that assumption.

The space $\lipfree{M}$ is a linearisation of $M$ in the following sense: it is a Banach space generated by an isometric copy $\delta(M)$ of $M$ that is linearly dense, with the property that every Lipschitz function $f:M\to X$ into a Banach space $X$ such that $f(0)=0$ can be extended to a bounded linear operator $F:\lipfree{M}\to X$ (identifying $M$ with $\delta(M)$), with $\norm{F}$ agreeing with the Lipschitz constant of $f$. Similarly, every base point-preserving Lipschitz map $f:M\to M'$ between metric spaces can be linearised into an operator $\widehat{f}:\lipfree{M}\to\lipfree{M'}$. In particular, taking $X=\RR$ shows that $\lipfree{M}$ is canonically an isometric predual of $\Lip_0(M)$. In the sequel, any mention of the weak$^*$ topology of $\Lip_0(M)$ will refer to the one induced by $\lipfree{M}$. This topology agrees on $B_{\Lip_0(M)}$ with the topology of pointwise convergence; if $M$ is compact, this agrees with the topology of uniform convergence as well.

According to McShane's extension theorem \cite[Theorem 1.33]{Weaver2}, any Lipschitz function $f:A\to\RR$ defined on a closed subset $A\subset M$ can be extended to a function $F:M\to\RR$ that has the same Lipschitz constant. It follows that $\lipfree{A\cup\set{0}}$ can be isometrically identified with a subspace of $\lipfree{M}$, namely $\cl{\lspan}\,\delta(A)$. For a given element $m\in\lipfree{M}$, there is a smallest closed set $S\subset M$ such that $m\in\lipfree{S\cup\set{0}}$; this set is called the \emph{support} of $m$ and denoted $\supp(m)$. Its existence is not at all obvious and was established in \cite{AP_rmi,APPP_2020}. Supports are always separable, and they cannot contain the base point $0$ as an isolated point (see \cite[p. 6]{APPP_2020}).

\subsection{De Leeuw representations}
\label{sec:de leeuw}

We define the set
$$
\wt{M} = \set{(x,y)\in M\times M \,:\, x\neq y}
$$
endowed with the subspace topology from $M\times M$, and for any function $f:M\to\RR$ and $(x,y)\in\wt{M}$ we denote
$$
\Phi f(x,y) = \frac{f(x)-f(y)}{d(x,y)} .
$$
For $f\in\Lip_0(M)$, $\Phi f$ is a bounded continuous real-valued function on $\wt{M}$, and can therefore be extended uniquely to a continuous function on $\bwt{M}$, the Stone-\v{C}ech compactification of $\wt{M}$; this extension will still be denoted $\Phi f$. This defines an operator $\Phi:\Lip_0(M)\to C(\bwt{M})$ that is a linear isometric embedding. Its adjoint $\Phi^*:C(\bwt{M})^*\to\Lip_0(M)^*$ is therefore a non-expansive, surjective quotient operator. The term \emph{De Leeuw transform} is used to refer to both maps $\Phi$ and $\Phi^*$.

We denote by $\meas{X}$ the space of signed Radon measures on the Hausdorff space $X$, and we say that a measure $\mu\in\meas{X}$ is concentrated on a set $A\subset X$ if $\mu(E)=\mu(A\cap E)$ for all $E\subset X$. In any such expression, we tacitly assume that $A$ and $E$ are Borel sets. Recall that $\meas{X}=C(X)^*$ when $X$ is compact, thus $\Phi^*$ maps $\meas{\bwt{M}}$ onto $\Lip_0(M)^*$. When $M$ is complete, both $M$ and $\wt{M}$ are completely metrisable and hence Borel (in fact, $G_\delta$) subsets of $\beta M$ and $\bwt{M}$, respectively (see \cite[Theorem 24.13]{Willard}).

Since $\Phi^*$ is onto, for any $m\in\Lip_0(M)^*$ (in particular, for $m\in\lipfree{M}$) there exist measures $\mu\in\meas{\bwt{M}}$ such that $\Phi^*\mu=m$, that is, such that
$$
\duality{f,m} = \int_{\bwt{M}}\Phi f\,d\mu \quad\text{for any $f\in\Lip_0(M)$.}
$$
Any such $\mu$ will be called a \emph{De Leeuw representation}, or simply \emph{representation}, of $m$. We always have $\norm{\mu}\geq\norm{m}$ but, because $\Phi^*$ is an adjoint quotient map, it is always possible to find representations of $m$ with minimal norm $\norm{m}$. It is not hard to see that $\mu$ can even be chosen to be positive \cite[Proposition 3]{Aliaga}. If a representation is positive and has minimal norm, we will say that it is \emph{optimal}. We denote the set of all optimal representations of $m\in\Lip_0(M)^*$ by
$$
\opr{m} = \set{\mu\in\meas{\bwt{M}} \,:\, \text{$\Phi^*\mu=m$, $\norm{\mu}=\norm{m}$ and $\mu\geq 0$}} ,
$$
and the set of all optimal representations (of any functional) by
$$
\opr{\bwt{M}} = \set{\mu\in\meas{\bwt{M}} \,:\, \text{$\norm{\Phi^*\mu}=\norm{\mu}$ and $\mu\geq 0$}}.
$$

Clearly $\opr{\bwt{M}}$ is norm-closed and $\opr{m}$ is weak$^*$ compact.
Another straightforward but very useful fact is that the set $\opr{\bwt{M}}$ is closed downward: if $\mu\in\meas{\bwt{M}}$ is optimal then any $\mu'\in\meas{\bwt{M}}$ with $0\leq\mu'\leq\mu$ is also optimal \cite[Proposition 2.1]{APS1}. In particular, restrictions of optimal representations to Borel subsets of $\bwt{M}$ are again optimal and provide a simple convex decomposition of functionals as
$$\norm{\Phi^*\mu}=\norm{\Phi^*(\mu\restrict_A)}+\|\Phi^*(\mu\restrict_{\bwt{M}\setminus A})\|$$
whenever $\mu\in\opr{\bwt{M}}$ and $A\subset \bwt{M}$ is Borel. (Recall that the restriction $\mu\restrict_A$ is defined by $\mu\restrict_A(E)=\mu(A\cap E)$ for $E\subset\bwt{M}$). Note also that if $0\in A\subset M$ and $A$ is closed then every (optimal) representation $\mu\in\meas{\bwt{A}}$ of an element $m\in\lipfree{A}$ is also a (optimal) representation of $m$ as an element of $\lipfree{M}$, as $\bwt{A}$ can be identified with the closure of $\wt{A}$ in $\bwt{M}$.

As $\wt{M}\subset M\times M$, points of $\wt{M}$ naturally carry a pair of coordinates. In order to extend this to $\bwt{M}$, we consider the identity mapping $\wt{M}\to M\times M$ and extend it to a continuous mapping $\pp:\bwt{M}\to\beta M\times\beta M$. Its first and second projections will be denoted by $\pp_i:\bwt{M}\to\beta M$, $i=1,2$. For a given set $E\subset\bwt{M}$, we will write $\pp_s(E)=\pp_1(E)\cup\pp_2(E)$ and call that subset of $\beta M$ the \emph{shadow} of $E$. Shadows are relevant because they allow us to relate the support of a functional to that of its representations. Indeed, if $m\in\lipfree{M}$ then any representation $\mu\in\meas{\bwt{M}}$ of $m$ satisfies $\supp(m)\subset\pp_s(\supp(\mu))$ \cite[Lemma 8]{Aliaga}. We will also occasionally consider the marginals of measures $\mu\in\meas{\bwt{M}}$, defined as pushforwards $(\pp_i)_\sharp\mu\in\meas{\beta M}$, $i=1,2$, given by $(\pp_i)_\sharp\mu(E)=\mu(\pp_i^{-1}(E))$ for $E\subset\beta M$. These marginals satisfy $\int_{\bwt{M}}(\varphi\circ\pp_i)\,d\mu=\int_{\beta M}\varphi\,d((\pp_i)_\sharp\mu)$ for any integrable Borel function $\varphi:\beta M\to\RR$. 

We make an important remark here: in the paragraph above, we may choose any compactification of $M$ in place of $\beta M$. In fact, the uniform compactification $\ucomp{M}$ is always a superior choice \cite{APS2}. However, the choice of compactification is irrelevant when dealing exclusively with functionals in $\lipfree{M}$, as is the case in this paper, so we will stick to $\beta M$ for simplicity and to avoid extra definitions.

\subsection{Choquet theory of Lipschitz-free spaces}
\label{sec:choquet}

We now provide a short introduction to the theory developed in \cite{Smith24}. The main idea is the introduction of a quasi-order on the set $\meas{\bwt{M}}^+$ of positive Radon measures on $\bwt{M}$ that compares how ``efficient'' they are, in some sense, at representing their corresponding functional. In order to do this, a convex cone $G\subset C(\bwt{M})$ is defined as the set of all functions $g \in C(\bwt{M})$ satisfying
\begin{equation}\label{G-function}
d(x,y)g(x,y) \leq d(x,u)g(x,u) + d(u,y)g(u,y) \quad\text{whenever $x,u,y \in M$ are distinct.}
\end{equation}
It is straightforward to check that $\mathbf{1}_{\bwt{M}} \in G$ and that $\Phi f\in G$ whenever $f \in \Lip_0(M)$; in fact, $\Phi(\Lip_0(M))=G \cap (-G)$.
A quasi-order $\preccurlyeq$ is then defined on $\meas{\bwt{M}}^+$ by declaring
$$
\mu\preccurlyeq\nu \quad\text{whenever}\quad \duality{g,\mu} \leq \duality{g,\nu} \text{ for all $g\in G$,}
$$
where $\duality{g,\mu}$ denotes the integral $\int_{\bwt{M}} g\,d\mu$. 

The motivation behind this definition is illustrated best by means of the following example \cite[Example 3.5]{Smith24}.
Let $M=[0,1]$ (with base point $0$) have the usual metric and consider the measures $\delta_{(1,0)}$ and
$$
\nu = \tfrac{1}{2}(\delta_{(1,\frac{1}{2})} + \delta_{(\frac{1}{2},0)})
$$
on $\wt{M}$. It is easily seen that both are optimal representations of $\delta(1)\in\lipfree{M}$. Moreover, $\delta_{(1,0)} \preccurlyeq \nu$ because
$$
\duality{g,\delta_{(1,0)}} = g(1,0) \leq \tfrac{1}{2}g(1,\tfrac{1}{2}) + \tfrac{1}{2}g(\tfrac{1}{2},0) = \duality{g,\nu}
$$
for all $g \in G$. Intuitively, $\delta_{(1,0)}$ can be viewed as a ``better'' representation of $\delta(1)$ than $\nu$ because it is a single Dirac measure on $\bwt{M}$, while $\nu$ is a combination of two Diracs. By passing from $\nu$ to $\delta_{(1,0)}$, the shared coordinate $\frac{1}{2}$ in the Diracs that comprise $\nu$ has been removed. It is this idea of elimination of shared coordinates that the quasi-order $\preccurlyeq$ is designed to capture.

The relation $\preccurlyeq$ is an analogue of the classical Choquet order. In classical Choquet theory, the focus lies on the measures that are maximal with respect to the Choquet order. Instead, we concentrate on those measures that are \emph{minimal} with respect to $\preccurlyeq$. While $\preccurlyeq$ is reflexive and transitive, it fails to be anti-symmetric because $G$ does not separate points of $\bwt{M}$ in general \cite[Example 3.2]{Smith24}. Thus minimality is defined as follows: we say that $\mu \in \meas{\bwt{M}}^+$ is minimal if $\nu \preccurlyeq \mu$ implies $\mu \preccurlyeq \nu$ for every $\nu \in \meas{\bwt{M}}^+$.

A standard application of Zorn's lemma demonstrates that minimal representations are plentiful. Indeed, for any $\nu \in \meas{\bwt{M}}^+$ there exists a minimal $\mu \in \meas{\bwt{M}}^+$ such that $\mu \preccurlyeq \nu$ \cite[Proposition 3.10]{Smith24}. All such $\mu$ satisfy $\Phi^*\mu = \Phi^* \nu$ and, moreover, if $\nu$ is optimal then so is $\mu$ \cite[Proposition 3.6]{Smith24}. In addition, the relation $\preccurlyeq$ favours some desirable subsets of $\bwt{M}$ over others. For instance, minimal measures try to avoid the set
$$
d^{-1}(0)=\set{\zeta\in\bwt{M} : d(\zeta)=0} ,
$$
which can be interpreted as a sort of ``long diagonal'' in $\bwt{M}$ (here and below, $d:\bwt{M}\to [0,\infty]$ is the continuous extension of $d:\wt{M}\to (0,\infty)$). By trying to avoid $d^{-1}(0)$, we mean $\mu(d^{-1}(0))\leq\nu(d^{-1}(0))$ whenever $\mu\preccurlyeq\nu$ \cite[Proposition 3.18]{Smith24}. In addition, if $\nu$ is concentrated on certain ``good'' sets such as $\wt{M}$, $\pp^{-1}(M\times M)$, or $\bwt{M}\setminus d^{-1}(0)$, then so is $\mu$ \cite[Proposition 3.18 and Corollary 3.27]{Smith24}.

Measures that are both optimal and minimal have a host of agreeable properties. For the purposes of this paper, which focuses on representations of functionals in $\lipfree{M}$, the crucial one is that any such representation must be concentrated on $\pp^{-1}(M\times M)$ \cite[Theorem 5.2]{Smith24}. That is, intuitively, in order to represent such a functional we only need to use coordinates of $M$, not its compactification. Moreover, we can replace $M$ with the support of the functional, possibly adding the base point \cite[Corollary 5.6]{Smith24}.

Some of the key points from the above discussion can be summarised by means of the following statement.

\begin{theorem}[Smith; cf. {\cite[Corollary 5.6]{Smith24}}]\label{th:opt_conc}
Let $m \in \lipfree{M}$. Then there exists $\mu \in \opr{m}$ concentrated on $\pp^{-1}(M\times M)$. Moreover, we can choose it to be concentrated on $\pp^{-1}(A\times A)$, where $A=\supp(m)\cup\set{0}$.
\end{theorem}

Most of the results in this paper depend on the theory developed in \cite{Smith24} only through Theorem \ref{th:opt_conc}. Additional properties will only be used in Section \ref{sec:diagonal}.

\section{A compact decomposition principle}

Our first main application of the theory established in \cite{Smith24} is the following decomposition theorem. It will also be the key tool used to prove our other main results, Theorems \ref{th:extreme} and \ref{th:p1u_cim}.

\begin{theorem}\label{th:compact_decomposition}
For every $m \in \lipfree{M}$, there exists a sequence $(m_n)$ in $\lipfree{M}$ such that
\begin{equation}\label{eq:compact_decomposition}
m = \sum_{n=1}^\infty m_n \quad\text{,}\quad \norm{m} = \sum_{n=1}^\infty \norm{m_n} \quad\text{and}\quad \text{$\supp(m_n)$ is compact for every $n$.}
\end{equation}
Moreover, the $m_n$ can be chosen so that $\supp(m_n)\subset\supp(m)$.
\end{theorem}

\begin{proof}
The statement will follow immediately by repeated application of Proposition \ref{pr:inner regularity} below.
\end{proof}

This theorem relies crucially on Theorem \ref{th:opt_conc} and is, in fact, equivalent to it (the converse implication follows from an argument similar to that of Theorem \ref{th:p1u_cim}). Thus, we regard Theorem \ref{th:compact_decomposition} as a ``user-friendly'' version of Theorem \ref{th:opt_conc}. Theorem \ref{th:compact_decomposition} can also be understood as a weak form of inner regularity for elements of Lipschitz-free spaces (see Proposition \ref{pr:inner regularity} below), or as an isometric counterpart to the compact reduction principle from \cite{ANPP}.

\begin{remark}
The decomposition in Theorem \ref{th:compact_decomposition} is a convex sum, but it cannot be chosen to be an $\ell_1$ sum in general, that is, it does not necessarily hold that $\norm{\sum_n\varepsilon_nm_n} = \sum_n\norm{m_n}$ for every choice of signs $\varepsilon_n\in\set{1,-1}$. Indeed, suppose this were true. For any non-compact $M$, we can find $m\in\lipfree{M}$ with non-compact support, e.g. $m=\sum_n 2^{-n}\delta(x_n)/d(x_n,0)$ where $(x_n)$ is a sequence in $M\setminus\set{0}$ with no convergent subsequence. Then the representation \eqref{eq:compact_decomposition} of $m$ would contain infinitely many non-zero elements $m_n$, which would become an isometric $\ell_1$ basis after normalising. It would follow that $\lipfree{M}$ contains $\ell_1$ isometrically whenever $M$ is not compact. An example due to Ostrovska and Ostrovskii shows that this is not the case \cite[Theorem 3.1]{OO1}.
\end{remark}

Our way towards the proof of Theorem \ref{th:compact_decomposition} starts with the following lemma, which is an adaptation of \cite[Lemma 11]{Aliaga}. The original lemma assumes that $M$ is \emph{proper}, i.e. its closed balls are compact, which implies that optimal De Leeuw representations of elements of $\lipfree{M}$ are concentrated on $\pp^{-1}(M\times M)$ \cite[Proposition 7]{Aliaga}. It turns out that this is the only essential requirement for the argument. Thanks to Theorem \ref{th:opt_conc}, we will be able to find representations satisfying this assumption even in the non-compact case.

\begin{lemma}[Aliaga; cf. {\cite[Lemma 11]{Aliaga}}]\label{lm:Aliaga_variant}
Let $\mu\in \meas{\bwt{M}}$ be a representation of an element of $\lipfree{M}$. Suppose that $\mu$ is concentrated on $\pp^{-1}(M \times M)$. Let $h:M\to\RR$ be a Lipschitz function with bounded support and $i\in\set{1,2}$, and define $\lambda\in\meas{\bwt{M}}$ by $d\lambda = (h\circ\pp_i)\,d\mu$. Then $\lambda$ represents an element of $\lipfree{M}$.
\end{lemma}

Note that the function $h\circ\pp_i$ is not well-defined because $h$ is defined on $M$ while $\pp_i$ takes values in a compactification of $M$, so to be correct we should consider an extension of $h$. However, if $\mu$ is concentrated on $\pp^{-1}(M\times M)$ then the value of $h$ outside of $M$ is irrelevant, so we permit this abuse of notation in both Lemmas \ref{lm:Aliaga_variant} and \ref{lm:Aliaga_Borel} below.

The proof of Lemma \ref{lm:Aliaga_variant} is almost identical to that of \cite[Lemma 11]{Aliaga}. We reproduce the argument here for completeness (and take the opportunity to fix a small mistake in the original).

\begin{proof}
Without loss of generality, assume $i=1$. Define $\nu\in\meas{\bwt{M}}$ by $d\nu=(\Phi h)\,d\mu$. Since $h\circ\pp_1$ and $\Phi h$ are continuous and bounded, we have $\lambda,\nu\in\meas{\bwt{M}}$. Moreover, they are concentrated on $\pp^{-1}(M \times M)$ because $\mu$ is.
For any $f\in\Lip_0(M)$, we have $fh\in\Lip_0(M)$ and a straightforward computation shows that the identity
$$
\Phi (fh) = (\Phi f)\cdot (h\circ\pp_1) + (\Phi h)\cdot (f\circ\pp_2)
$$
holds in $\wt{M}$, hence also in $\pp^{-1}(M\times M)$ by continuous extension. Therefore
\begin{align*}
\duality{fh,\dual{\Phi}\mu} &= \int_{\pp^{-1}(M \times M)}(\Phi f)(h\circ\pp_1)\,d\mu + \int_{\pp^{-1}(M \times M)}(f\circ\pp_2)(\Phi h)\,d\mu \\
&= \duality{f,\dual{\Phi}\lambda} + \int_{\pp^{-1}(M\times M)}(f\circ\pp_2)\,d\nu \\
&= \duality{f,\dual{\Phi}\lambda} + \int_M f\,d((\pp_2)_\sharp\nu) .
\end{align*}
In order to prove that $\Phi^*\lambda\in\lipfree{M}$, it will suffice to show that the same holds for the functionals $f\mapsto\duality{fh,\dual{\Phi}\mu}$ and $f \mapsto \int_M f\,d((\pp_2)_\sharp\nu)$. The former follows e.g. from \cite[Lemma 2.3]{APPP_2020}. By \cite[Proposition 4.4]{AP_measures}, the latter is equivalent to finiteness of the integral
$$
\int_M d(x,0)\,d\abs{(\pp_2)_\sharp\nu}(x) \leq \int_M d(x,0)\,d((\pp_2)_\sharp\abs{\nu})(x) = \int_{\pp^{-1}(M\times M)} d(\pp_2(\zeta),0)\cdot\abs{\Phi h(\zeta)}\,d\abs{\mu}(\zeta)
$$
so, by continuity, it is enough to show that the function $g(x,y)=d(y,0)\abs{\Phi h(x,y)}$ is bounded for $(x,y)\in\wt{M}$. To that end, fix $r>0$ such that $\supp(h)\subset B(0,r)$. We consider three cases.
\begin{itemize}
\item If $y\in B(0,2r)$, then $g(x,y)\leq 2r\lipnorm{h}$.
\item If $y\notin B(0,2r)$ and $x\in B(0,r)$, then $h(y)=0$ and
$$
g(x,y) = \frac{d(y,0)}{d(x,y)}\cdot |h(x)| \leq \pare{1+\frac{d(x,0)}{d(x,y)}}\cdot |h(x)| \leq 2\cdot (r\lipnorm{h}+|h(0)|) .
$$
\item If $y\notin B(0,2r)$ and $x\notin B(0,r)$, then $h(y)=h(x)=0$ and $g(x,y)=0$.
\end{itemize}
So $g$ is bounded by $2(r\lipnorm{h}+|h(0)|)$ on $\wt{M}$, and this finishes the proof.
\end{proof}

Next, we extend Lemma \ref{lm:Aliaga_variant} to non-Lipschitz weighting functions $h$.

\begin{lemma}\label{lm:Aliaga_Borel}
Let $\mu\in\meas{\bwt{M}}$ be a representation of an element of $\lipfree{M}$. Suppose that $\mu$ is concentrated on $\pp^{-1}(M \times M)$. Let $h:M\times M\to\RR$ be a Borel function such that
$$
\int_{\bwt{M}} \abs{h}\circ\pp\, d\abs{\mu} < \infty ,
$$
and define $\lambda\in\meas{\bwt{M}}$ by $d\lambda = (h\circ\pp)\,d\mu$. Then $\lambda$ represents an element of $\lipfree{M}$. If moreover $\mu$ is optimal and $h\geq 0$, then $\lambda$ is optimal.
\end{lemma}

\begin{proof}
Note that $\lambda$ is really a Radon measure on $\bwt{M}$ as its total variation $\norm{\lambda} \leq \int_{\bwt{M}}\abs{h\circ\pp}d\abs{\mu}$ is finite. Also, because of the regularity of $\mu$ we may assume that $M$ is separable by replacing it with $M'=\cl{\bigcup_n \pp_s(K_n)}$, where $(K_n)$ is a sequence of compact subsets of $\pp^{-1}(M\times M)$ such that $\mu$ is concentrated on $\bigcup_n K_n$.

By Lemma \ref{lm:Aliaga_variant}, $\Phi^*\lambda\in\lipfree{M}$ when $h\circ\pp=g\circ\pp_i$ for some Lipschitz function $g:M\to\RR$ with bounded support. We will now extend this fact to more general classes of functions $h$ incrementally.

Suppose first that $h\circ\pp=\mathbf{1}_K\circ\pp_1$ for some compact set $K\subset M$. Given $n \in \NN$, define $h_n:M\to[0,1]$ by $h_n(x)=\max\set{0,1-nd(x,K)}$. These functions are Lipschitz, have bounded support, and decrease pointwise and monotonically to $\mathbf{1}_K$ on $M$. Define measures $\lambda_n\in\meas{\bwt{M}}$ by $d\lambda_n=(h_n \circ \pp_1)\,d\mu$. Then 
$$
\norm{\Phi^*\lambda_n-\Phi^*\lambda} \leq \norm{\lambda_n-\lambda}
\leq \int_{\pp^{-1}(M\times M)} (h_n-h) \circ \pp_1 \,d\abs{\mu} ,
$$
which converges to $0$ by the dominated convergence theorem. Therefore $\Phi^*\lambda_n\to \Phi^*\lambda$ in norm, but $\Phi^*\lambda_n\in\lipfree{M}$ for all $n$ by Lemma \ref{lm:Aliaga_variant}, so $\Phi^*\lambda\in \lipfree{M}$ as well.

Next, suppose that $h\circ\pp=\mathbf{1}_E\circ\pp_1$ for some Borel set $E\subset M$. Fix $\varepsilon>0$ and, by regularity, find a compact $K\subset E$ with $(\pp_1)_\sharp\abs{\mu}(E\setminus K)<\varepsilon$. Define $\lambda'\in\meas{\bwt{M}}$ by $d\lambda'=(\mathbf{1}_K\circ\pp_1)\,d\mu$. Then $\Phi^*\lambda'\in\lipfree{M}$ by the previous case and
$$
\norm{\Phi^*\lambda-\Phi^*\lambda'} \leq \norm{\lambda-\lambda'} \leq \int_{\bwt{M}} (\mathbf{1}_E-\mathbf{1}_K)\circ\pp_1\,d\abs{\mu} = (\pp_1)_\sharp\abs{\mu}(E\setminus K) < \varepsilon
$$
so $\dist(\Phi^*\lambda,\lipfree{M})<\varepsilon$. Since $\varepsilon$ was arbitrary, it follows that $\Phi^*\lambda\in\lipfree{M}$.

Clearly, the previous argument is valid with $\pp_2$ in place of $\pp_1$. Hence, if $E,F$ are Borel subsets of $M$ then $h=\mathbf{1}_{E\times F}$ also yields $\Phi^*\lambda\in\lipfree{M}$ as $d\lambda = (\mathbf{1}_{E\times F}\circ\pp)\,d\mu = (\mathbf{1}_E\circ\pp_1)\cdot(\mathbf{1}_F\circ\pp_2)\,d\mu$.

Now assume that $h=\mathbf{1}_K$ where $K\subset M\times M$ is closed. Let $U=(M\times M)\setminus K$. Since $M$ is separable, $U$ can be written as the union of countably many basic open sets of the form $V_n\times W_n$, where $V_n,W_n\subset M$ are open. Therefore
$$
\mathbf{1}_K = 1 - \mathbf{1}_U = \prod_{n=1}^\infty (1 - \mathbf{1}_{V_n\times W_n})
$$
pointwise on $M\times M$. Define inductively $\lambda_0=\mu$ and
$$
d\lambda_{n+1} = (1-(\mathbf{1}_{V_n\times W_n}\circ\pp))\,d\lambda_n = (1-(\mathbf{1}_{V_n}\circ\pp_1)\cdot(\mathbf{1}_{W_n}\circ\pp_2))\,d\lambda_n
$$
for $n\geq 0$, then $\Phi^*\lambda_n\in\lipfree{M}$ for every $n$ by the previous cases and induction, and
$$
\norm{\Phi^*\lambda-\Phi^*\lambda_n} \leq \norm{\lambda-\lambda_n} \leq \int_{\pp^{-1}(M\times M)}\left(\mathbf{1}_K - \prod_{j=1}^n (1 - \mathbf{1}_{V_j\times W_j})\right)\circ\pp \,d|\mu|
$$
converges to $0$ by the dominated convergence theorem. Therefore $\Phi^*\lambda\in\lipfree{M}$.

The case $h=\mathbf{1}_E$ for Borel $E\subset M\times M$ follows from the previous one in the same way as it did for one dimension.

Finally, suppose that $h$ is a Borel function. Then $h\in L_1(\pp_\sharp\abs{\mu})$ by hypothesis so, for any $\varepsilon>0$, we can find a simple function $s=\sum_{k=1}^n a_k\mathbf{1}_{E_k}$ for some $n\in\NN$, $a_k\in\RR$, and Borel $E_k\subset M\times M$, such that $\norm{h-s}_{L_1(\pp_\sharp\abs{\mu})} < \varepsilon$ and therefore
$$
\norm{\Phi^*\lambda - \Phi^*\lambda'} \leq \norm{\lambda-\lambda'} \leq \int_{\bwt{M}} \abs{h-s}\circ\pp\,d\abs{\mu} = \int_{M\times M} \abs{h-s}\,d(\pp_\sharp\abs{\mu}) < \varepsilon
$$
where $d\lambda'=(s\circ\pp)\,d\mu$. Note that $\Phi^*\lambda'\in\lipfree{M}$ by the previous case and linearity. Therefore $\dist(\Phi^*\lambda,\lipfree{M})<\varepsilon$ and we conclude again that $\Phi^*\lambda\in\lipfree{M}$.

We finish by verifying the last assertion. Suppose that $\mu$ is optimal. If $h$ is positive and bounded, then $0\leq\lambda\leq c\cdot\mu$ for some $c<\infty$ and therefore $\lambda\in\opr{\bwt{M}}$ by \cite[Proposition 2.1]{APS1}. In particular, $\lambda$ is optimal when $h$ is a positive simple function. The result now follows from the previous paragraph, as $\opr{\bwt{M}}$ is norm-closed and we can choose $s\geq 0$ if $h\geq 0$.
\end{proof}

If we set $\nu=\pp_\sharp\abs{\mu} \in \meas{M\times M}$ and define a linear map $T:L_1(\nu)\to \meas{\bwt{M}}$ by $d(Th) = (h\circ\pp)\,d\mu$, then $T$ is clearly an isometry. Lemma \ref{lm:Aliaga_Borel} shows that the range of $T$ consists of representations of elements of $\lipfree{M}$; if moreover $\mu$ is optimal, then $Th$ is optimal whenever $h \geq 0$. Similar statements hold for $\nu_i=(\pp_i)_\sharp\abs{\mu}\in\meas{M}$ and $T_i:L_1(\nu_i)\to\meas{\bwt{M}}$ defined by $d(T_ih)=(h\circ\pp_i)\,d\mu$, $i=1,2$.

\medskip

We can finally prove the following proposition, which establishes a weak inner regularity-like property of elements of Lipschitz-free spaces and is instrumental for proving Theorem \ref{th:compact_decomposition}. It crucially relies on Theorem \ref{th:opt_conc}.

\begin{proposition}\label{pr:inner regularity}
Let $m \in \lipfree{M}$ and $\ep>0$. Then there exists a compact set $K \subset \supp(m)$ and an element $m' \in \lipfree{M}$ supported on $K$, such that $\norm{m}=\norm{m'} + \norm{m-m'}$ and $\norm{m-m'} < \ep$.
\end{proposition}

\begin{proof}
By Theorem \ref{th:opt_conc}, there exists $\mu\in\opr{m}$ concentrated on $\pp^{-1}(A \times A)$, where $A=\supp(m)\cup\{0\}$. By regularity of $\mu$, find a compact set $\mathcal{K}\subset\pp^{-1}(A\times A)$ with $\mu(\pp^{-1}(A\times A)\setminus\mathcal{K})<\varepsilon$. Then $K=\pp_s(\mathcal{K})$ is a compact subset of $A$ and $\mathcal{K}\subset\pp^{-1}(K\times K)\subset\pp^{-1}(A\times A)$. The measure $\lambda=\mu\restrict_{\pp^{-1}(K\times K)}$ is optimal, and can also be described by $d\lambda=(\mathbf{1}_{K\times K}\circ\pp)\,d\mu$. Then $m'=\Phi^*\lambda$ belongs to $\lipfree{M}$ by Lemma \ref{lm:Aliaga_Borel}. Since $0\leq\lambda\leq\mu$, we have $\mu,\lambda,\mu-\lambda \in \opr{\bwt{M}}$ and therefore
\[
\norm{m} = \norm{\mu} = \norm{\lambda} + \norm{\mu-\lambda} = \norm{m'} + \norm{m-m'}
\]
where
$$
\norm{m-m'} = \norm{\mu-\lambda} = \mu(\pp^{-1}(A\times A)\setminus \pp^{-1}(K\times K)) \leq \mu(\pp^{-1}(A\times A)\setminus \mathcal{K}) < \varepsilon .
$$
Finally, we have $\supp(m')\subset K$ by \cite[Lemma 8]{Aliaga}, but $0$ cannot be an isolated point in either $\supp(m)$ or $\supp(m')$, so we conclude $\supp(m')\subset\supp(m)$.
\end{proof}

We finish this section by providing an alternative statement for part of the argument in Lemma \ref{lm:Aliaga_Borel} and Proposition \ref{pr:inner regularity}, expressed in terms of representations concentrated on the diagonal of $\bwt{M}$. To that end, we introduce the notation
\begin{equation}\label{eq:Delta notation}
\Delta(E) = \set{ \zeta\in\bwt{M} \,:\, \pp(\zeta)=(x,x) \text{ for some } x\in E }  = \pp^{-1}(E \times E) \cap d^{-1}(0)
\end{equation}
for subsets $E\subset M$. This notation will be used in the later sections as well. Note that $\pp^{-1}(M\times M)=\wt{M}\cup\Delta(M)$, so any De Leeuw representation concentrated on $\pp^{-1}(M\times M)$ can be decomposed into one concentrated on the diagonal $\Delta(M)$ and another one concentrated on $\wt{M}$. The latter always corresponds to an element of $\lipfree{M}$ by \cite[Proposition 2.6]{APS1}.

\begin{corollary}
Suppose that $\mu\in\meas{\bwt{M}}$ is concentrated on $\Delta(M)$ and $\Phi^*\mu\in\lipfree{M}$. Then $\Phi^*(\mu\restrict_{\Delta(E)})\in\lipfree{M}$ for every Borel $E\subset M$.
\end{corollary}

\begin{proof}
Since $\Delta(M)\subset\pp^{-1}(M\times M)$, we can apply Lemma \ref{lm:Aliaga_Borel} to show that $\lambda=\mu\restrict_{\Delta(E)}$, which satisfies $d\lambda = (\mathbf{1}_E\circ\pp_1)\,d\mu$, represents an element of $\lipfree{M}$.
\end{proof}

\section{The extreme point problem}
\label{sec:extreme}

We shall now use Theorem \ref{th:compact_decomposition} to obtain a characterisation of the extreme points of the unit ball of $\lipfree{M}$ for arbitrary $M$. This problem dates back to the pioneering work of Weaver in the 1990s. Since then, the primary suspects have been the \emph{(elementary) molecules}, that is, elements of $\lipfree{M}$ of the form
$$
m_{xy} = \frac{\delta(x)-\delta(y)}{d(x,y)}
$$
for $(x,y)\in\wt{M}$. We denote the set of all elementary molecules in $\lipfree{M}$ by $\Mol(M)$. Note that molecules are the canonical norming elements for $\Lip_0(M)$, so $B_{\lipfree{M}}=\cl{\conv}\,\Mol(M)$ by the Hahn-Banach theorem; hence the suspicion. Not all molecules are extreme points: for any choice of different points $x,p,y$ in $M$ we have
$$
m_{xy} = \frac{d(x,p)}{d(x,y)}m_{xp} + \frac{d(p,y)}{d(x,y)}m_{py}
$$
so $m_{xy}$ is not extreme if there exists $p\neq x,y$ such that $d(x,p)+d(p,y)=d(x,y)$. The first two named authors proved that the converse implication also holds for complete $M$ \cite{AP_rmi}; this characterises extreme molecules completely.

The problem is thus determining whether all extreme points must be molecules. This was first asked explicitly in \cite[Question 2]{AliagaGuirao}, \cite[p. 2]{GPPR}, and \cite[p. 102]{Weaver2} at around the same time, but is implicit in Weaver's early work such as \cite{Weaver95} or the first edition of \cite{Weaver2}, where he used De Leeuw representations to prove that \emph{preserved} extreme points of $B_{\lipfree{M}}$ (meaning that they are also extreme in $B_{\lipfree{M}^{**}}$) must be molecules \cite[Corollary 3.44]{Weaver2}, and that all extreme points are molecules when $M$ is compact and satisfies an additional condition related to $\lipfree{M}$ being a dual Banach space (see \cite[Corollary 4.41]{Weaver2}).

These first results, and their approach based on De Leeuw representations, have had a profound impact on subsequent developments. Recent work has significantly expanded the list of known cases where $\ext B_{\lipfree{M}} \subset \Mol(M)$ to include all compact spaces \cite{Aliaga}, uniformly discrete spaces \cite{APS1}, and subsets of $\RR$-trees \cite{APP}. Stronger versions of extreme points have also been completely characterised for all $M$, including preserved extreme points \cite{AliagaGuirao}, denting points \cite{GPPR}, and strongly exposed points \cite{GPAZ}.

Theorem \ref{th:compact_decomposition} now allows us to reduce the extreme point problem to the compact case and provide a definitive answer.

\begin{theorem}\label{th:extreme}
Let $M$ be a complete metric space. Then $\ext B_{\lipfree{M}} \subset \Mol(M)$. Specifically, the extreme points of $B_{\lipfree{M}}$ are exactly the molecules $m_{xy}$ with $x\neq y\in M$ such that $d(x,p)+d(p,y)>d(x,y)$ for every $p\in M\setminus\set{x,y}$.
\end{theorem}

\begin{proof}
Let $m\in\ext B_{\lipfree{M}}$, and find an expression $m=\sum_n m_n$ as in \eqref{eq:compact_decomposition}. By convexity, we have $m=m_n/\norm{m_n}$ for all $n$ such that $m_n\neq 0$. Since there must be at least one such $n$, $m$ has compact support. The result now follows from \cite[Corollary 12]{Aliaga} and \cite[Theorem 1.1]{AP_rmi}.
\end{proof}

\begin{remark}
It is also possible to bypass Theorem \ref{th:compact_decomposition} and prove Theorem \ref{th:extreme} directly from Theorem \ref{th:opt_conc} and Lemma \ref{lm:Aliaga_variant} by a similar argument to that used in \cite{Aliaga} for the proper case. 
Indeed, in the proof of \cite[Theorem 1]{Aliaga}, given $m \in \ext B_{\lipfree{M}}$, we make use of a representation $\mu \in \opr{\bwt{M}}$ of $m$ that is concentrated on $\pp^{-1}(M \times M)$,  for which we can fix $\zeta \in \supp(\mu) \cap \pp^{-1}(M \times M)$ by inner regularity. By weighting the measure $\mu$ we then conclude that $m$ is supported in $\{\pp_1(\zeta),\pp_2(\zeta)\}\subset M$, and hence $m$ is a molecule. If $M$ is proper then every optimal representation of $m$ is concentrated on $\pp^{-1}(M \times M)$ \cite[Proposition 7]{Aliaga}. Without the assumption that $M$ is proper, we can use Theorem \ref{th:opt_conc} to find some such $\mu$. Once such a $\mu$ has been fixed, the proof of \cite[Theorem 1]{Aliaga} can be followed exactly as it is written, except that we use Lemma \ref{lm:Aliaga_variant} whenever the original proof requires \cite[Lemma 11]{Aliaga}.
\end{remark}

Let us spell out some consequences of Theorem \ref{th:extreme}.
Recall that $x \in B_X$ is an \emph{exposed point} of the unit ball of the Banach space $X$ if there exists $f \in S_{X^*}$ such that $\duality{f,x} =1 > \duality{f,y}$ whenever $y \in B_X\setminus\{x\}$; any exposed point of $B_X$ is clearly extreme. It is shown in \cite[Theorem 3.2]{APPP_2020} that all extreme molecules are in fact exposed points of $B_{\lipfree{M}}$, so the following is immediate.

\begin{corollary}\label{cr:exposed}
Every extreme point of $B_{\lipfree{M}}$ is also an exposed point.
\end{corollary}

Another consequence is a functional characterisation of geodesic metric spaces. Recall that $M$ is called \emph{geodesic} if, given $x,y \in M$, there exists an isometric embedding $\gamma$ of $[0,d(x,y)] \subset \RR$ into $M$, such that $\gamma(0)=x$ and $\gamma(d(x,y))=y$. Given that $M$ is assumed to be complete, this is equivalent to the property that for every $(x,y) \in \wt{M}$ there exists $p\in M\setminus\set{x,y}$ such that $d(x,p)+d(p,y)=d(x,y)$ \cite[Proposition 4.1]{GPAZ}. Thus we obtain the following result straightaway.

\begin{corollary}\label{cr:geodesic}
If $M$ is complete, then $\ext B_{\lipfree{M}}=\varnothing$ if and only if $M$ is geodesic.
\end{corollary}

Theorem \ref{th:extreme} also has some implications for surjective linear isometries between Lipschitz-free spaces. We refer the reader to \cite{AFGLZ_20,CDT_24} and \cite[Section 3.8]{Weaver2} for more information on this emerging topic. Note that any surjective linear isometry from a Banach space $X$ onto a Banach space $Y$ maps $\ext B_X$ onto $\ext B_Y$, so Theorem \ref{th:extreme} ought to yield some information about them. We consider two extreme cases. First, if $\ext B_X$ is empty then $\ext B_Y$ must be empty as well. Thus Corollary \ref{cr:geodesic} yields:

\begin{corollary}\label{cr:isometries}
Let $M$ and $M'$ be complete metric spaces such that $\lipfree{M}$ and $\lipfree{M'}$ are linearly isometric. If $M$ is geodesic, then so is $M'$.
\end{corollary}

The opposite situation occurs when $\ext B_{\lipfree{M}}$ is as big as possible, that is, when all molecules are extreme points. This happens precisely when $M$ is \emph{concave}, meaning that the triangle inequality is always strict for triples of different points in $M$. In that case, we can obtain a Lipschitz-free version of the classical Banach-Stone theorem. So that we can state the result properly, we review three types of transformations on $M$ that induce trivial isometries on its Lipschitz-free space:
\begin{itemize}
\item Scaling the metric with a constant multiplicative factor $c>0$ yields an isometry of Lipschitz-free spaces $S_c:\lipfree{M,d}\to\lipfree{M,c\cdot d}$ given by $S_c(m)=\frac{1}{c}m$, $m\in\lipfree{M,d}$.
\item If $M'$ is another metric space and $\pi:M\to M'$ is a base point-preserving, surjective isometry, then its linearisation $\widehat{\pi}:\lipfree{M}\to\lipfree{M'}$ is a surjective linear isometry.
\item Choosing an alternative base point $0' \in M$ also induces an isometry $J_{0'}:\lipfree{M,0'}\to\lipfree{M,0}$ given by $\duality{f,J_{0'}m}=\duality{f-f(0')\mathbf{1}_M,m}$, where $m \in \lipfree{M,0'}$ and $f \in \Lip_0(M)$.
\end{itemize}
All three can be combined as follows. We call a surjective map $\pi:(M,d)\to (M',d')$ between metric spaces a \emph{$c$-dilation} (where $c>0$) if $d'(\pi(x),\pi(y))=c\cdot d(x,y)$ for all $x,y \in M$. Then any such map induces a linear isometry $T:\lipfree{M,d,0}\to\lipfree{M',d',0'}$ given by $T=J_{\pi(0)}\circ\widehat{\pi}\circ S_c$. For concave metric spaces, these trivial isometries are the only possible ones.

\begin{theorem}\label{th:banach-stone}
Let $M$ and $M'$ be complete, concave metric spaces and let $T:\lipfree{M}\to\lipfree{M'}$ be a surjective linear isometry. Then there exists a surjective $c$-dilation $\pi:M\to M'$, such that $T=\pm\frac{1}{c}J_{\pi(0)}\circ\widehat{\pi}$.
\end{theorem}

In particular, $M$ and $M'$ are isometric up to a constant multiplicative factor. The theorem fails if we remove the assumption of concavity, as e.g. $\lipfree{M}$ is isometric to $\ell_1$ whenever $M\subset\RR$ is closed and Lebesgue null \cite{Godard}.

\begin{proof}[Sketch of proof] The proof of Theorem \ref{th:banach-stone} follows Weaver's arguments in \cite[Section 3.8]{Weaver2}, where he derives the conclusion from a stronger hypothesis of uniform concavity. The crux of his argument is the following fact, which we state independently.

\begin{proposition}[Weaver; cf. {\cite[Lemma 3.54 and Theorem 3.55]{Weaver2}}]
Let $M$ and $M'$ be complete metric spaces, and suppose that there exists a surjective linear isometry $T:\lipfree{M}\to\lipfree{M'}$ such that $T(\Mol(M))=\Mol(M')$. Then there exists a surjective $c$-dilation $\pi:M\to M'$, such that $T=\pm\frac{1}{c}J_{\pi(0)}\circ\widehat{\pi}$.
\end{proposition}

\noindent He then applies this to spaces where $\Mol(M)$ agrees with the set of \emph{preserved} extreme points of $B_{\lipfree{M}}$, which is always preserved by linear isometries. Doing the same with spaces where $\Mol(M)=\ext B_{\lipfree{M}}$ yields Theorem \ref{th:banach-stone}.
\end{proof}

The recent work \cite{CDT_24} by C\'uth, Doucha and Titkos also derives information on isometries between Lipschitz-free spaces by identification and matching of their preserved extreme points. It is conceivable that Theorem \ref{th:extreme} could lead to straightforward extensions of the results in \cite{CDT_24}, as with Weaver's theorem. We have not explored this possibility.

\section{Convex integrals of molecules}
\label{sec:cims}

A functional $m\in\Lip_0(M)^*$ is called a \emph{convex integral of molecules} if it admits an optimal De Leeuw representation $\mu$ that is concentrated on $\wt{M}$. Any such functional belongs to $\lipfree{M}$ and can, in fact, be expressed as a literal Bochner integral of molecules $m=\int_{\wt{M}} m_{xy}\,d\mu(x,y)$ \cite[Proposition 2.6]{APS1}. These elements generalise convex combinations, or series, of molecules, which correspond precisely to the case where $\mu$ is a discrete measure.

Using the Choquet theory from \cite{Smith24}, it can be proved that if $0\in A\subset M$ and $m\in\lipfree{A}$ is a convex integral of molecules in $\lipfree{M}$, then it is also a convex integral of molecules in $\lipfree{A}$ \cite[Corollary 4.12]{Smith24} (that is, if it has an optimal representation concentrated on $\wt{M}$, then it can be chosen so that it is concentrated on $\wt{A}$). The converse is clearly also true. Thus, being a convex integral of molecules does not depend on the ambient space.

In \cite[Problem 6.9 (b)]{APS2}, we asked which metric spaces $M$ have the property that every element of $\lipfree{M}$ is a convex integral of molecules. It follows from the remark above that this is a hereditary property. Known cases include uniformly discrete spaces \cite[Corollary 3.7]{APS1} and compact, purely 1-unrectifiable spaces \cite[Corollary 6.6]{APS2}. We recall that a metric space $M$ is \emph{purely 1-unrectifiable} if, for every $E\subset\RR$ and Lipschitz map $f:E\to M$, the image $f(E)$ has null 1-Hausdorff measure. Equivalently, $M$ is purely 1-unrectifiable when it contains no \emph{curve fragment}, i.e. a bi-Lipschitz copy of a compact subset of $\RR$ with positive measure (see \cite[Corollary 1.12]{AGPP}). Pure 1-unrectifiability characterises the Radon-Nikod\'ym property for Lipschitz-free spaces \cite[Theorem C]{AGPP}.

As a second consequence of Theorem \ref{th:compact_decomposition}, we can now generalise all known cases and show that all complete purely 1-unrectifiable spaces are solutions to \cite[Problem 6.9 (b)]{APS2}.

\begin{theorem}\label{th:p1u_cim}
Let $M$ be a complete purely 1-unrectifiable metric space. Then every element of $\lipfree{M}$ is a convex integral of molecules.
\end{theorem}

In the proof, we will use the following simple fact that we state independently for later reference.

\begin{lemma}\label{lm:cim l1 sum}
Suppose that $m\in\lipfree{M}$ is the convex sum of finitely or countably many elements $m_n\in\lipfree{M}$. If each $m_n$ is a convex integral of molecules, then so is $m$.
\end{lemma}

\begin{proof}
For each $n$, choose $\mu_n\in\opr{m_n}$ concentrated on $\wt{M}$. Then $\mu=\sum_n\mu_n$ converges absolutely, with $\norm{\mu}=\sum_n\norm{\mu_n}=\sum_n\norm{m_n}=\norm{m}$. Moreover $\Phi^*\mu=\sum_n\Phi^*\mu_n=\sum_nm_n=m$, so $\mu\in\opr{m}$. Since each $\mu_n$ is concentrated on $\wt{M}$, so is $\mu$, thus $m$ is a convex integral of molecules.
\end{proof}

\begin{proof}[Proof of Theorem \ref{th:p1u_cim}]
Let $m\in\lipfree{M}$, and find an expression $m=\sum_n m_n$ as in \eqref{eq:compact_decomposition}, so that $S_n=\supp(m_n)\cup\set{0}$ is compact for all $n$. Since $S_n$ is also purely 1-unrectifiable, $m_n$ is a convex integral of molecules in $\lipfree{S_n}$ by \cite[Corollary 6.6]{APS2}. Thus there exists an optimal representation $\mu_n\in\opr{\bwt{S_n}}$ of $m_n$ concentrated on $\wt{S_n}$, that can clearly be identified with an optimal representation $\mu_n\in\opr{\bwt{M}}$ of $m_n$ concentrated on $\wt{S_n}\subset\wt{M}$. That is, each $m_n$ is also a convex integral of molecules in $\lipfree{M}$. The theorem now follows from Lemma \ref{lm:cim l1 sum}.
\end{proof}

As a particular case, \cite[Proposition 2.8]{APS1} yields the next result.

\begin{corollary}
If $M$ is complete and scattered (in particular if it is countable or discrete) then every element of $\lipfree{M}$ is a convex series of molecules.
\end{corollary}

\begin{remark}
It is also possible to derive Theorem \ref{th:p1u_cim} from Theorem \ref{th:extreme} as follows. We only sketch the argument. Let $m\in S_{\lipfree{M}}$. By replacing $M$ with $\supp(m)\cup\{0\}$, we may assume that $M$ is separable. By \cite[Theorem C]{AGPP}, $\lipfree{M}$ has the Radon-Nikod\'ym property. This allows us to apply Edgar's non-compact version of Choquet's theorem \cite{Edgar}, according to which every element of $S_{\lipfree{M}}$ can be expressed as a Bochner probability integral on any measurable subset of $S_{\lipfree{M}}$ containing $\ext B_{\lipfree{M}}$. Since $\ext B_{\lipfree{M}}\subset\Mol(M)$ by Theorem \ref{th:extreme} and $\Mol(M)$ is a norm-closed subset of $S_{\lipfree{M}}$, the integral can be taken, in particular, on $\Mol(M)$. It only remains to observe that $\Mol(M)$ is homeomorphic to $\wt{M}$ via the map $m_{xy}\mapsto (x,y)$. So, that expression corresponds to a convex integral of molecules.
\end{remark}

We do not know any example of a non-purely 1-unrectifiable space satisfying the hypotheses of \cite[Problem 6.9]{APS2}, and we believe it to be a reasonable conjecture that the converse of Theorem \ref{th:p1u_cim} could hold. A very similar question was posed in \cite[Question 4.2]{APS1} (we will show in Section \ref{sec:diagonal} that it is in fact an equivalent question). Given \cite[Corollary 4.13]{Smith24}, the question can be asked in the following equivalent form.

\begin{question}\label{q:curve fragment}
Let $K$ be a curve fragment, i.e. a bi-Lipschitz copy of a compact subset of $\RR$ with positive measure. Does there exist an element of $\lipfree{K}$ that is not a convex integral of molecules?
\end{question}

Next, we derive a consequence of Theorem \ref{th:p1u_cim}.

\begin{corollary}\label{cr:p1u sna}
Let $M$ be a complete purely 1-unrectifiable metric space. If $f\in\Lip_0(M)$ attains its norm as a functional on $\lipfree{M}$, then it attains its Lipschitz constant.
\end{corollary}

\begin{proof}
Let $m\in S_{\lipfree{M}}$ be such that $\duality{f,m}=\lipnorm{f}$. By Theorem \ref{th:p1u_cim}, $m$ is a convex integral of molecules, meaning that $m=\Phi^*\mu$ for some probability measure $\mu\in\meas{\beta\wt{M}}$ concentrated on $\wt{M}$. Then we have
$$
\lipnorm{f} = \duality{f,m} = \int_{\wt{M}} \Phi f \, d\mu \leq \int_{\wt{M}} \lipnorm{f} \, d\mu = \lipnorm{f} .
$$
The inequality is thus an equality, therefore $\Phi f=\lipnorm{f}$ $\mu$-almost everywhere on $\wt{M}$. In particular, there exists $(x,y)\in\wt{M}$ such that $\Phi f(x,y)=\lipnorm{f}$, which is exactly our desired conclusion.
\end{proof}

The set of functions in $\Lip_0(M)$ that attain their Lipschitz constant is sometimes denoted $\mathrm{SNA}(M)$. It is proved in \cite[Theorem 7.4]{GPPR} that $\mathrm{SNA}(M)$ is dense in $\Lip_0(M)$ whenever $\lipfree{M}$ has the Radon-Nikod\'ym property. Corollary \ref{cr:p1u sna} establishes a stronger property (density follows from the Bishop-Phelps theorem). The converse of \cite[Theorem 7.4]{GPPR} doesn't hold: \cite[Theorem 2.5]{CGMRZ} showcases a complete metric space $M$ containing an isometric copy of $[0,1]$ and such that $\cl{\mathrm{SNA}(M)}=\Lip_0(M)$. But we do not know whether the converse of Corollary \ref{cr:p1u sna} is true.

\begin{question}\label{q:sna}
Suppose that $\lipfree{M}$ fails the Radon-Nikod\'ym property. Does there exist a function in $\Lip_0(M)$ that attains its norm on $\lipfree{M}$ but not its Lipschitz constant?
\end{question}

The argument used for the proof of Corollary \ref{cr:p1u sna} shows that any $f\in\Lip_0(M)$ that attains its norm at a convex integral of molecules must attain its Lipschitz constant. Thus, a positive answer to Question \ref{q:sna} would imply a positive answer to Question \ref{q:curve fragment} as well. It is known that Question \ref{q:sna} has a positive answer when $M$ is (the range of) a $C^1$ curve \cite[Theorem 1.2]{Chiclana}.

\medskip

Even when $M$ is not purely 1-unrectifiable, a ``local'' version of Theorem \ref{th:p1u_cim} remains true. Recall the notation $\Delta(E)$ for $E\subset M$ introduced in \eqref{eq:Delta notation}. It is shown in \cite[Lemma 2.4]{APS1} that any optimal representation $\mu$ of an element of $\lipfree{M}$ has zero mass at $\Delta(\set{x})$ for any $x\in M$. It follows that $\mu(\Delta(E))=0$ for countable $E\subset M$ and, by regularity, also for scattered $E$. We now extend this fact to purely 1-unrectifiable sets.

\begin{proposition}\label{pr:diagonal_p1u}
Let $\mu\in\opr{\bwt{M}}$ be an optimal representation of an element of $\lipfree{M}$. Then $\mu(\Delta(E))=0$ for every purely 1-unrectifiable Borel subset $E$ of $M$.
\end{proposition}

Before proving Proposition \ref{pr:diagonal_p1u}, let us see how it yields yet another proof of Theorem \ref{th:p1u_cim}. By Theorem \ref{th:opt_conc}, every $m\in\lipfree{M}$ admits an optimal representation $\mu\in\opr{\bwt{M}}$ concentrated on $\pp^{-1}(M\times M)$. If $M$ is purely 1-unrectifiable, then $\mu(\Delta(M))=0$ by Proposition \ref{pr:diagonal_p1u} and therefore $\mu$ is concentrated on $\pp^{-1}(M\times M)\setminus\Delta(M)=\wt{M}$. In fact, this provides a slight strenghtening of Theorem \ref{th:p1u_cim}: \emph{every} optimal and minimal representation of an element of $\lipfree{M}$ is concentrated on $\wt{M}$.

For the proof, the notion of locally flat function will be required. Given $f:M\to\RR$, we say that $f$ is \emph{locally flat at $x\in M$} if the Lipschitz constant of $f\restrict_{B(x,r)}$ converges to $0$ as $r\to 0$; equivalently, if $\Phi f$ vanishes on $\Delta(\set{x})$. If $f$ is locally flat at every point of $M$, we simply say that it is \emph{locally flat}. The subspace of $\Lip_0(M)$ consisting of locally flat functions is denoted by $\lip_0(M)$.

One of the main results in \cite{AGPP} is that locally flat functions are weak$^*$ dense in $B_{\Lip_0(K)}$ when $K$ is compact and purely 1-unrectifiable. The next proposition extends this result to $\Lip_0(M)$ when $K\subset M$. It will be the basis for the proof of Proposition \ref{pr:diagonal_p1u}.

\begin{proposition}\label{pr:p1u compact approx}
Let $K$ be a compact, purely 1-unrectifiable subset of $M$. Then the set of 1-Lipschitz functions that are locally flat at every point of $K$ is weak$^*$ sequentially dense in $B_{\Lip_0(M)}$.
\end{proposition}

\begin{proof}
We may and will assume that $0\in K$. Recall that $\lip_0(K)$ is an isometric predual of $\lipfree{K}$ by \cite[Theorem B]{AGPP}, thus $B_{\lip_0(K)}$ is weak$^*$ dense in $B_{\Lip_0(K)}$. Fix $f\in S_{\Lip_0(M)}$. Since $K$ is compact, we can approximate $f\restrict_K$ with a sequence of functions $g_n\in B_{\lip_0(K)}$ such that $\abs{g_n(x)-f(x)}\leq 1/n^2$ for all $x\in K$. By the Di Marino-Gigli-Pratelli extension theorem \cite[Theorem 1.1]{DGP}, we may extend $g_n$ to a function $h_n\in\Lip_0(M)$ such that $h_n\restrict_K=g_n$, $\lipnorm{h_n}\leq 1+\tfrac{1}{n}$, and $h_n$ is locally flat at every point of $K$ \emph{with $M$ as the ambient space}, that is, $\Phi h_n=0$ on $\Delta(K)\subset\bwt{M}$.

Now define the function $f_n$ on the set $\set{x\in M : d(x,K)\leq 1/n^2 \text{ or } d(x,K)\geq 5/n}$ as
$$
f_n(x) = \begin{cases}
h_n(x) &\text{if $d(x,K)\leq 1/n^2$} \\
f(x) &\text{if $d(x,K)\geq 5/n$}.
\end{cases}
$$
If $d(x,K)\leq 1/n^2$ and $d(y,K)\geq 5/n$, then there exists $p\in K$ with $d(x,p)\leq 1/n^2$, so
$$
\abs{f_n(x)-f(x)} \leq \abs{h_n(x)-h_n(p)} + \abs{g_n(p)-f(p)} + \abs{f(p)-f(x)} \leq (1+\tfrac{1}{n})\tfrac{1}{n^2} + \tfrac{1}{n^2} + \tfrac{1}{n^2} \leq \tfrac{4}{n^2}
$$
and therefore
$$
\abs{f_n(x)-f_n(y)} \leq \abs{f_n(x)-f(x)} + \abs{f(x)-f(y)} \leq \tfrac{4}{n^2}+d(x,y) \leq \pare{1+\tfrac{1}{n}}d(x,y) .
$$
Thus we can extend $f_n$ to a $(1+\tfrac{1}{n})$-Lipschitz function defined on all of $M$. It is clear that the sequence $(f_n)$ converges to $f$ pointwise, hence in the weak$^*$ topology. Since every $f_n$ agrees with $h_n$ on a neighbourhood of $K$, it is locally flat at every point of $K$. By normalising each $f_n$, we can arrange for them to be 1-Lipschitz and still converge weak$^*$ to $f$.
\end{proof}

\begin{proof}[Proof of Proposition \ref{pr:diagonal_p1u}]
Let $m=\Phi^*\mu$. Fix a compact subset $\mathcal{K}$ of $\Delta(E)$ and note that $\mathcal{K}$ is contained in $\Delta(K)$ where $K=\pp_s(\mathcal{K})\subset E$ is a compact, purely 1-unrectifiable set. Fix some $f\in S_{\Lip_0(M)}$ such that $\duality{f,m}=\norm{m}$. By Proposition \ref{pr:p1u compact approx}, there exists a sequence $(f_n)$ in $B_{\Lip_0(M)}$ that converges weak$^*$ to $f$ and such that every $f_n$ is locally flat at every point of $K$, that is, $\Phi f_n=0$ on $\Delta(K)$. Hence
\begin{align*}
\norm{\mu} = \norm{m} = \duality{f,m} = \lim_{n\to\infty} \duality{f_n,m} &= \lim_{n\to\infty}\int_{\bwt{M}} \Phi f_n \,d\mu = \lim_{n\to\infty}\int_{\bwt{M}\setminus \Delta(K)} \Phi f_n \,d\mu \\
&\leq \lim_{n\to\infty}\lipnorm{f_n}\cdot\mu(\bwt{M}\setminus \Delta(K)) \\
&\leq \mu(\bwt{M}\setminus \Delta(K)) .
\end{align*}
It follows that $\mu(\Delta(K))=0$, so $\mu(\mathcal{K})=0$. Since $\mathcal{K}$ was arbitrary, the regularity of $\mu$ yields $\mu(\Delta(E))=0$ as desired.
\end{proof}

Compactness is essential to the proof of 
Proposition \ref{pr:p1u compact approx} and the results in \cite[Section 2]{AGPP} it depends on, but it is reasonable to expect the corresponding approximation result to hold in general (i.e. non-compact) purely 1-unrectifiable spaces, too.

\begin{question}
If $M$ is complete and purely 1-unrectifiable, can every 1-Lipschitz function $f:M\to\RR$ be approximated pointwise (hence weak$^*$) by locally flat 1-Lipschitz functions?
\end{question}

\section{Diagonal functionals}
\label{sec:diagonal}

Not every metric space satisfies the conclusion of Theorem \ref{th:p1u_cim}. To analyse what happens in general, we define a class of Lipschitz-free space elements whose (non-zero) members can be regarded as being as far from being convex integrals of molecules as possible.

\begin{definition}
We call $m \in \lipfree{M}$ a \emph{diagonal element} if $\mu(\wt{M})=0$ for every $\mu \in \opr{m}$.
\end{definition}

The term ``diagonal element'' acknowledges the fact that every optimal and minimal representation of such an element is concentrated on the diagonal $\pp^{-1}(M\times M)\setminus\wt{M}=\Delta(M)$. In fact, a stronger property holds: every optimal representation, not necessarily minimal, is concentrated on the ``long diagonal'' $d^{-1}(0)$. Here, we note that $\pp^{-1}(M\times M) \cap d^{-1}(0) = \Delta(M)$.

\begin{proposition}\label{pr:diagonal}
An element $m\in\lipfree{M}$ is diagonal if and only if every $\mu\in\opr{m}$ is concentrated on $d^{-1}(0)$.
\end{proposition}

Moreover, similarly to convex integrals of molecules, a functional that is diagonal with respect to a subset of $M$ is still diagonal with respect to $M$.

\begin{proposition}\label{pr:diagonal ambient}
Let $A \subset M$ be closed, with $0 \in A$, and let $m \in \lipfree{A}$. Then $m$ is diagonal as an element of $\lipfree{A}$ if and only if it is diagonal as an element of $\lipfree{M}$.
\end{proposition}

We will prove both statements simultaneously.

\begin{proof}[Proof of Propositions \ref{pr:diagonal} and \ref{pr:diagonal ambient}]
Both backward implications are clear: the former because $\wt{M}\cap d^{-1}(0)=\varnothing$, and the latter because every representation of $m$ in $\meas{\bwt{A}}$ can be identified with a representation in $\meas{\bwt{M}}$ and $\bwt{A}\cap\wt{M}=\wt{A}$. To complete the proof, we will show that if $m$ is diagonal in $\lipfree{A}$ then every $\nu\in\opr{\bwt{M}}$ such that $\Phi^*\nu=m$ must be concentrated on $d^{-1}(0) \subset\bwt{M}$.

We prove the contrapositive. Assume there exists $\nu\in\opr{\bwt{M}}$ with $\Phi^*\nu=m$, such that $\nu(d^{-1}(0))<\norm{\nu}=\norm{m}$. By \cite[Proposition 3.10]{Smith24}, there exists a minimal representation $\mu\preccurlyeq\nu$. Then $\mu$ is an optimal representation of $m$, $\mu$ is concentrated on $\pp^{-1}(A\times A)$, and $\mu(d^{-1}(0))\leq\nu(d^{-1}(0))$, by Proposition 3.6, Corollary 5.6, and Proposition 3.18 in \cite{Smith24}, respectively. This implies $\mu(d^{-1}(0)) < \norm{m} = \norm{\mu}$ and therefore $0 < \mu(\pp^{-1}(A\times A)\setminus d^{-1}(0)) = \mu(\wt{A})$.

Let $\lambda=\mu\restrict_{\wt{A}}$. Both measures $\lambda$ and $\mu-\lambda=\mu\restrict_{\bwt{M}\setminus\wt{A}}$ are optimal and, since $m'=\Phi^*\lambda\in\lipfree{A}$, we also have $m-m'=\Phi^*(\mu-\lambda)\in\lipfree{A}$. Moreover $\norm{m'}+\norm{m-m'}=\norm{\lambda}+\norm{\mu-\lambda}=\norm{\mu}=\norm{m}$. Find an optimal representation $\omega$ of $m-m'$ in $\meas{\bwt{A}}$. Then $\omega+\lambda$ is an optimal representation of $m$ in $\meas{\bwt{A}}$, as $\norm{\omega+\lambda}=\norm{\omega}+\norm{\lambda}=\norm{m-m'}+\norm{m'}=\norm{m}$. Moreover $(\omega+\lambda)(\wt{A})\geq\lambda(\wt{A})>0$. Therefore $m$ is not a diagonal element in $\lipfree{A}$.
\end{proof}

The only element of $\lipfree{M}$ that is both diagonal and a convex integral of molecules is $0$. Non-zero diagonal elements are known to exist, e.g. when $M$ contains an isometric copy of a subset of $\RR$ with positive measure \cite[Theorem 4.1]{APS1}. In fact, the following result implies that such elements must exist as soon as there are elements of $\lipfree{M}$ that cannot be written as a convex integral of molecules.

\begin{theorem}\label{th:diagonal decomposition}
Every $m\in\lipfree{M}$ can be expressed as $m=m_c+m_d$ where $m_c,m_d\in\lipfree{M}$ are such that $\norm{m}=\norm{m_c}+\norm{m_d}$, $m_c$ is a convex integral of molecules, and $m_d$ is diagonal. Moreover, they can be chosen so that $\supp(m_c),\supp(m_d)\subset\supp(m)$.
\end{theorem}

It is possible to prove Theorem \ref{th:diagonal decomposition} using a (somewhat) elementary exhaustion argument and Zorn's lemma, constructing sequences of convex integrals of molecules whose optimal representations are totally ordered. But the resulting proof is long and cumbersome so, instead, we shall give a quick proof that relies on the theory and methods developed in \cite{Smith24}; this further illustrates its power.

We require the next auxiliary result which yields some straightforward examples of functions in the cone $G$ defined in \eqref{G-function}.

\begin{lemma}\label{lm:inc_subadd}
Let the continuous map $\tau:[0,\infty]\to\RR$ have the property that $t \mapsto t\tau(t)$ is increasing and subadditive on $[0,\infty)$. Then $\tau \circ d \in G$.
\end{lemma}

\begin{proof}
Let $g=\tau \circ d \in C(\bwt{M})$. Given distinct $x,u,y \in M$, by the properties of $\tau$ we have
\begin{align*}
d(x,y)g(x,y) = d(x,y)\tau(d(x,y)) &\leq (d(x,u)+d(u,y))\cdot\tau(d(x,u)+d(u,y)) \\
&\leq d(x,u)\tau(d(x,u)) + d(u,y)\tau(d(u,y)) \\
&= d(x,u)g(x,u) + d(u,y)g(u,y). \qedhere
\end{align*}
\end{proof}

\begin{proof}[Proof of Theorem \ref{th:diagonal decomposition}]
Fix $m\in\lipfree{M}$ and set $S=\supp(m)$ and $A=S\cup\set{0}$.
Define $g:\bwt{M}\to[0,1]$ by
\[
  g(\zeta) = \begin{cases} \dfrac{1}{1+d(\zeta)} & \text{if }d(\zeta) < \infty \\ 0 & \text{if }d(\zeta)=\infty.  \end{cases}
\]
As $t \mapsto t/(1+t)$ is increasing and subadditive on $[0,\infty)$, $g \in G$ by Lemma \ref{lm:inc_subadd}. Since $\opr{m}$ is $w^*$-compact and $\mu \mapsto \duality{g,\mu}$ is $w^*$-continuous, there exists $\mu \in \opr{m}$ such that $\duality{g,\mu} \leq \duality{g,\nu}$ for all $\nu \in \opr{m}$. The same holds for every $\mu'\preccurlyeq\mu$ because $g\in G$, so we may choose $\mu$ to be minimal (here we use \cite[Propositions 3.6 and 3.10]{Smith24}). By \cite[Corollary 5.6]{Smith24}, $\mu$ is concentrated on $\pp^{-1}(A\times A)$. Define $\mu_d = \mu\restrict_{d^{-1}(0)}$ and $m_d = \Phi^*(\mu_d)$. Then $\mu_d$ and $\mu-\mu_d=\mu\restrict_{\wt{A}}$ are optimal, so $m_c=m-m_d$ is a convex integral of molecules and
$$
\norm{m_c}+\norm{m_d}=\norm{\mu-\mu_d}+\norm{\mu_d}=\norm{\mu}=\norm{m} .
$$

Let us show that $m_d$ is diagonal. Assume otherwise, then there exists $\mu'\in\opr{m_d}$ such that $\mu'(\wt{M})>0$. Since $g<1$ on $\wt{M}$, we have
$$
\duality{g,\mu'}<\norm{\mu'}=\norm{m_d}=\norm{\mu_d}=\duality{g,\mu_d} ,
$$
where the last equality holds because $\mu_d$ is concentrated on $d^{-1}(0)$. Thus $\lambda=(\mu-\mu_d)+\mu'$ satisfies $\Phi^*\lambda=m$, so $\lambda\in\opr{m}$ and
\[
\duality{g,\lambda} = \duality{g,\mu-\mu_d} + \duality{g,\mu'} < \duality{g,\mu-\mu_d} + \duality{g,\mu_d} = \duality{g,\mu},
\]
contradicting the minimality of $\mu$ with respect to $g$. Therefore $m_d$ is diagonal.

Finally, we verify that $m_c,m_d$ are supported on $S$. The measure $\mu_d\in\opr{m_d}$ is concentrated on $\pp^{-1}(A\times A)$, and hence supported on $\pp^{-1}(\cl{A}\times\cl{A})$, where the closure of $A$ is taken in the compactification of $M$. Then \cite[Lemma 8]{Aliaga} implies that $\supp(m_d)\subset \cl{A}\cap M=A$. Moreover, recall that if $0$ belongs to the support of an element of $\lipfree{M}$, then it cannot be an isolated point of the support. It follows that $0\in S$ whenever $0\in\supp(m_d)$, so $\supp(m_d)\subset S$. Since $m_c=m-m_d$, we get that $\supp(m_c)\subset\supp(m)\cup\supp(m_d)=S$ as well.
\end{proof}

We do not know if the decomposition in Theorem \ref{th:diagonal decomposition} is unique. If so, we can at least say that it is not linear. Indeed, the two following examples show that it is possible for the sum of two convex integrals of molecules to be a diagonal element and vice versa. Both examples are constructed in the Lipschitz-free space over $M=[0,1]$ with $0$ as the base point.

\begin{example}
Let $(a_n,b_n)$, $n \in \NN$, be pairwise disjoint open subintervals of $[0,1]$ having dense union and satisfying $\sum_{n=1}^\infty (b_n-a_n) < 1$. Define
$$
m=\sum_{n=1}^\infty (b_n-a_n)m_{a_nb_n} \quad\text{and}\quad \mu = \sum_{n=1}^\infty (b_n-a_n)\delta_{(a_n,b_n)} .
$$
Then $\mu \in \opr{m}$, so $m$ is a convex integral of molecules. However, by following the proof of \cite[Theorem 4.1]{APS1}, we see that $m_{10}+m$ is a non-zero diagonal element of $\lipfree{M}$.
\end{example}

\begin{example}
Let $\lambda$ denote Lebesgue measure on $[0,1]$, and let $A\subset [0,1]$ be a Borel set such that $0 < \lambda(A \cap [a,b]) < b-a$ whenever $0 \leq a < b \leq 1$ (see \cite{Rudin83}). Let $T:L_1([0,1])\to\lipfree{[0,1]}$ be the usual isometry, given by $\duality{f,Th}=\int_0^1 f'(t)h(t)\,dt$ for $f \in \Lip_0([0,1])$, $h \in L_1([0,1])$. Set $m=T\mathbf{1}_A$ and $m'=T\mathbf{1}_{[0,1]\setminus A}$. Then $m+m'=T\mathbf{1}_{[0,1]}=m_{10}$. Reasoning as in \cite[Theorem 4.1]{APS1} shows that both $m$ and $m'$ are diagonal. Indeed, $\norm{m}=\norm{\mathbf{1}_A}_1=\lambda(A)$, and thus the function $f \in B_{\Lip_0([0,1])}$ given by $f(x)=\int_0^x \mathbf{1}_A(t) \,dt = \lambda(A \cap [0,x])$, $x \in [0,1]$ norms $m$. Now let $\mu \in \opr{m}$ and $(x,y) \in \wt{M}$. If $x < y$ then $\Phi f(x,y) \leq 0$, and if $x>y$ then
\[
\Phi f(x,y) = \frac{f(x)-f(y)}{x-y} = \frac{\lambda(A \cap (y,x])}{x-y} < 1.
\]
By \cite[Lemma 2.3]{APS1}, $\supp(\mu) \subset (\Phi f)^{-1}(1) \subset d^{-1}(0)$ and therefore $m$ is diagonal.  Given the properties of $A$, the argument showing that $m$ is diagonal applies equally to $m'$.
\end{example}

\section*{Acknowledgements}

Some of this research was carried out during visits of the first and third authors to the Faculty of Information Technology at the Czech Technical University in Prague in 2023 and 2024, and a visit of the second author to the School of Mathematics and Statistics at University College Dublin in 2023.

R. J. Aliaga was partially supported by Grant PID2021-122126NB-C33 funded by MICIU/AEI/ 10.13039/501100011033 and by ERDF/EU. E. Perneck\'a was supported by the Czech Science Foundation (GA\v CR) grant 22-32829S.

\end{document}